\documentclass[12pt]{article}

\usepackage{amsmath}
\usepackage{fixltx2e}
\usepackage{etex}
\usepackage{xspace}
\usepackage{lmodern}
\usepackage[T1]{fontenc}
\usepackage{textcomp}
\usepackage[utf8]{inputenc}
\usepackage{microtype}
\usepackage{hyperref}

\newcommand*{\cs}[1]{\texttt{\textbackslash#1}}
\makeatletter
\newcommand*{\cmd}[1]{\cs{\expandafter\@gobble\string#1}}
\makeatother

\usepackage{amssymb}
\usepackage{amsthm}
\usepackage{amsmath}

\newtheorem{definition}{Definition}
\newtheorem{thm}{Theorem}
\newtheorem{remark}{Remark}

\newtheorem{lemma}{Lemma}
\newtheorem{corollary}{Corollary}

\begin{document}

\title{On the maximal saddle order of $p:-q$ resonant saddle}

\author{Guangfeng Dong \\{\small Department of Mathematics, Jinan University,}\\{\small Guangzhou 510632, China,
donggf@jnu.edu.cn(Corresponding author)}\\ \\Changjian Liu \\ {\small School of Mathematics(Zhuhai), Sun Yat-Sen University,}\\{\small Zhuhai 519082, China,  liuchangj@mail.sysu.edu.cn}\\ \\
Jiazhong Yang \\ {\small School of Mathematical Sciences, Peking University, }\\{\small Beijing 100871, China, jyang@math.pku.edu.cn} }

\date{}

\maketitle

\begin{abstract}
In this paper, we obtain some estimations of the saddle order which is the sole topological invariant
of the non-integrable resonant saddles of planar polynomial vector fields of arbitrary degree $n$.
Firstly, we prove that, for any given resonance $p:-q$, $(p, q)=1$, and  sufficiently big integer $n$,
the maximal saddle order  can grow at least as rapidly as $n^2$.
Secondly, we show  that there exists an integer $k_0$, which grows at least as rapidly as $3n^2/2$,
 such that $L_{k_0}$ does not belong to the ideal
generated by the first $k_0-1$ saddle values $L_1, L_2, \cdots, L_{k_0-1}$,
where  $L_{k}$ means the $k$-th saddle value of the given system.
In particular,  if  $p=1$ (or $q=1$), we obtain a sharper result  that  $k_0$ can grow at least as rapidly as $2  n^2$.
\end{abstract}





\section{Introduction and  main results}
\label{intro}

Consider real planar polynomial ordinary differential equations:
\begin{equation}\label{rcf}\dot x=-y+f_n(x,y), \quad  \dot y=x+g_n(x,y),\end{equation}
where $f_n(x,y)$ and $g_n(x,y)$  are polynomials of degree $n$ consisting of nonlinear terms only. It is well known that
such a system always has a center or a focus at the origin, and to obtain criteria to distinguish them
is one of the most classical problems in the qualitative theory of ordinary differential equations.

Let $z=x+y\mathrm i$. We can transform  the   above system to a  complex form   with $1:-1$ resonance saddle
\begin{equation}\label{ccf}
\dot z=\mathrm i z+P_n(z, \bar z), \quad \dot {\bar z}=-\mathrm i \bar z+\bar P_n(z, \bar z),\end{equation}
where $P_n$ is a polynomial of its variables.
Then system (\ref{rcf}) is a center if and only if system (\ref{ccf}) is integrable, i.e.,
it has a first integral of the form $H(z, \bar z)=z\bar z+h.o.t.$,
where ``$h.o.t.$'' stands for ``higher-order terms''.

In this paper,  we concentrate on  a more general version of this problem.
Consider the following $p:-q$ resonance saddle system
 \begin{eqnarray}\label{C-S}
\dot{x}=px+P(x,y), \qquad \dot{y}=-q y+Q(x,y),
\end{eqnarray}
where $p$ and $q$ are  positive integers such that $(p,q)=1$, and $P$ and  $Q$ are real or complex polynomials of degree
$n$ having no constant and linear terms.
Without loss of generality, we assume $p\leq q$.
The origin is said to be a {\it generalized center} if there exists an analytic first integral in a neighborhood of the origin.
The corresponding system is said to be integrable.

To decide whether the origin of system (\ref{C-S}) is a generalized center,
naturally one needs to seek for an analytic first integral of  the form
 $H=x^{q}y^{p}+h.o.t.$  and study its derivative along the vector field.
 Namely, one can calculate the successive terms in the Taylor expansion of
 $H$ and $\dot H$,
 $$H=x^{q}y^{p}+h.o.t. , \qquad \dot{H}=  \sum_{k=1}^{\infty}L_{k}(x^{q}y^{p})^{k+1}.$$
 Obviously  system (\ref{C-S}) is integrable if and only if all of the values $L_{k}$ vanish.
 Therefore these coefficients $L_{k}$ play a critical role in the generalized center problem.

The value $L_k$ is called the $k$-th saddle value of the system (an alternative  definition is given in Section \ref{Pre}) and $k$ plays a role of describing the
order of the saddle value $L_k$.
These $L_k$ are not unique in general,
but the order of the first nonzero value $L_k$,
which is called the saddle order,
is an invariant of the system under the change of the form $(x,y) \mapsto (x+h.o.t.,y+h.o.t.)$.
Furthermore, it is also the sole topological invariant of the singularity if $p$ and $q$ are given(see \cite{CS-Inv}).
Besides, the saddle order is closely related with other problems, for example,  the cyclicity of a saddle or a homoclinic loop (see, e.g \cite{JoyRou,Zo2}).

To get all the conditions of a generalized center,
usually we must calculate all these saddle values one by one.
There are some  algebraic mechanisms  to derive them, which can be found in,
e.g  \cite{RomanShafer}. Using this algorithm, it is easy to prove that every saddle value is a polynomial
of the coefficients of $P$ and $Q$.

Denote by $I_{k}$($I_{\infty}$, resp.) the ideal generated by the first $k$ saddle values (all the saddle values, resp.)
in the ring consisting of all polynomials of the coefficients of $P$ and $Q$ over the field  $\mathbb R$ or $\mathbb{C}$.
According to Hilbert's Basis Theorem, we know that $I_{\infty}$ is finitely generated. Namely, there exists a  minimal  number $M=M(p,q,n)$ such that
$L_{k}=0$ for  all $k >M$ provided that only  $L_{k}=0$ for $k \leq M$.
In other words,
$M(p,q,n)$ is exactly the maximal possible  saddle order of the systems. Notice that for real and complex systems,
$M(p,q,n)$ can be different, but our results stand  for both cases.

Hilbert's Basis Theorem, however, only guarantees  the existence of such an $M(p,q,n)$,
it says  nothing about how to determine this number.
In fact, for any given tuple $(p,q,n)$,
to obtain such  a number $M(p,q,n)$ is an extremely difficult problem. Up to now,  systematically solved  cases  include
only $M(1,1,2)=3$ (see e.g. \cite{Bau}) and $M(1,2,2)=5$ (see e.g. \cite{FSZ}). We  even do not know the corresponding number $M(1,1,3)$.

Beyond the above systematic results, there are some partially known cases.
If  $P$ and $Q$ in (\ref{C-S}) are homogenous cubic polynomials, then
$M_{h}(1,1,3)=5$, $M_h(1,3,3)=8$ (see e.g  \cite{Sib2,Liu-YLi-J,HuRomanShafer}), where  $M_h(p,q,n)$
 with a subscript $h$
 indicates the maximal possible  saddle order of (\ref{C-S}) with homogenous  nonlinearity
of degree $n$.

Since it seems too difficult to decide the exact $M(p,q,n)$ or $M_h(p,q,n)$ for general $p,q,n$, it is quite natural to
examine the complexity of the problem by
looking for the lower bounds of them.
The known results, besides the mentioned exact ones, can be summarized as follows:
\begin{eqnarray*}
&& M(1,1,3)\geq 12,\quad
M_{h}(1,1,4)\geq 18,\quad  M_{h}(1,1,5)\geq 22,\\
&&M(1,3,2)\geq 6,\quad  M_{h}(1,2,3)\geq 8,\ \\
&&M_h(1,1,n)\geq
\begin{cases}
  n^2-1,\quad\ \   n \   \mbox{is even}\\
 \frac{1}{2}(n^2-1),\  n \  \mbox{is odd}
\end{cases}, \\
&&M_h(1,2,n)\geq n^2 -1, \   \mbox{one of}\  n+2\  \mbox{and}\  2n+1\  \mbox{is  prime}.
\end{eqnarray*}
These results and more can be found in \cite{Sadovskii,YuTian,huang,FerChenRoman,M132,QiuYang,bai,dong,G-L,G-V}.

In this paper, we shall study the polynomial systems
(\ref{C-S}) with any given resonant $p:-q$ saddle and arbitrary degree $n$.
Since we are more interested in the  tendency  of the saddle order   as  $n$ tends to infinity,         therefore
in this paper,
we assume that the number $n$, the degree of the system,  is sufficiently big, e.g., $n\gg p+q$.
As far as we know, these theorems obtained in this paper are the
first ones to consider the very general tuple $(p,q,n)$.

The first theorem of the paper deals with the case where
  $P$ and $Q$ are homogenous polynomials. We have the following

\begin{thm}\label{TM1}
For any given resonance $p:-q$ and  sufficiently big $n$,
the following inequality holds
$$M_h(p,q,n)\geq \frac{n^2-1}{d},$$
where  $d=(n-1, p+q)$.
\end{thm}

\begin{remark}
Notice that the number  $d$ plays a very important role in the integrability of the system.
Roughly speaking, the smaller  $d$ is, the more difficult  the problem  becomes.
 For example,  consider the following Lotka-Volterra system
\begin{equation*}\label{2}
\dot x=x(1-a_{0}x^2-a_{1}xy-a_{2}y^2), \quad \dot y=y(-q+b_{0}x^2+b_{1}xy+b_{2}y^2). \end{equation*}
We have $p=1$, $n=3$. If $q$ is odd, then $d=2$, the problem is comparatively simple.
In fact, in \cite{LiuLi-Yun}, it is proved that
this system is integrable if and only if its first three saddle values are zero.
If $q$ is even, it follows that $d=1$,  then  the situation turns out to be  very unpleasant.
The integrability of this system is not entirely solved even for $q=2$ (see e.g. \cite{CGRS}).
\end{remark}

The next theorem deals with a more general setting, i.e.,  we do not restrict system (\ref{C-S}) to homogenous nonlinearities. In this case, we prove that the maximal saddle order $M(p,q,n)$ can always grow
at least as rapidly as $n^2$ even when $d>1$.

\begin{thm}\label{TM2}
For any given  saddle  admitting  $p:-q$ resonance and for any  sufficiently big $n$,
the following inequality holds
$$M(p,q,n)\geq n^2-2rn+r^2-1,$$
where $0\leq r\leq p+q-1$ satisfying
$r\equiv n \, \,\mbox{mod}\  (p+q)$.
\end{thm}

 Due to  Hilbert's Basis Theorem, the ideal $I_{\infty}$ generated by all the saddle value $L_k$ is
 finitely generated. Denote by $M^{I}(p,q,n)$ the finite minimal  number  such that $I_{M^{I}}=I_{\infty}$,
i.e., $M^{I}(p,q,n)$ is exactly the maximal possible order of $L_{k}$ satisfying
$L_{k}\not \in I_{k-1}$.
Correspondingly, denote by $M^{I}_{h}(p,q,n)$  the value of $M^{I}(p,q,n)$ restricted to systems (\ref{C-S}) where $P$ and $Q$
are homogeneous polynomials of degree $n$.
Then in the following theorems, we  present some estimations of $M^{I}(p,q,n)$ and $M^{I}_{h}(p,q,n)$.

\begin{thm}\label{TMideal}
For system (\ref{C-S}) with any given resonance $p:-q$ and homogenous $P$ and $Q$ of degree $n$,         $n \gg 1$,
$$M_{h}^{I}(p,q,n)\geq \frac{3}{2d} n^{2}+O(n).$$
In particular, if  $p=1$, then
$$M_{h}^{I}(1,q,n)\geq \frac{2 }{d} n^{2}+O(n).$$
\end{thm}

Parallel to  Theorem \ref{TM2}, if we do not restrict system (\ref{C-S}) to homogenous nonlinearities,
then  we have

\begin{thm}\label{TMideal2}
For system (\ref{C-S}) with nonhomogeneous $P$ and $Q$, $n \gg 1$,
$$M^{I}(p,q,n)\geq \frac{3}{2} n^{2}+O(n).$$
In particular, if  $p=1$,    then
$$M^{I}(1,q,n)\geq 2 n^{2}+O(n).$$
\end{thm}

Obviously $M(p,q,n)\leq M^{I}(p,q,n)$.  However, the problem
 whether and/or when they  coincide with each other is quite open.

The paper is organized as follows:  In Section 2, we intend to provide some preliminaries
 such as definitions, notation and some lemmas.
 Then in Section 3 we shall present  detailed proof of all the theorems.

\section{Preliminaries}
\label{Pre}

First of all, we  recall some basic facts from normal form theory
of vector fields. For details, we recommend the readers \cite{Bibikov}.

It is well known that for system (\ref{C-S})
 there always exists a formal change of
coordinates
\begin{eqnarray}
\label{change} X=x+\sum_{k\geq2}\Phi_{k}(x,y),\quad
Y=y+\sum_{k\geq2}\Psi_{k}(x,y),
\end{eqnarray}
where $\Phi_{k}(x,y)=\sum_{i+j=k}\varphi_{i,j}x^{i}y^{j}$ and
$\Psi_{k}(x,y)=\sum_{i+j=k}\psi_{i,j}x^{i}y^{j}$ are homogeneous
polynomials of degree $k$,
 transforming system (\ref{C-S}) to its formal normal form
\begin{eqnarray*} \label{formal-nf}
\dot{X}=pX (1+\sum_{k\geq1}\pi_{1,k}(X^{q}Y^{p})^{k}),\quad
\dot{Y}=-q Y(1+\sum_{k\geq1}\pi_{2,k}(X^{q}Y^{p})^{k}).
\end{eqnarray*}

If we denote by $\pi_k=\pi_{1,k}-\pi_{2,k}$, then the
minimal number $k$  such that   $\pi_k \neq 0$  is an invariant of the system under the change of the form (\ref{change}).

\begin{definition} \label{def-quan}
The quantity $\pi_k$ is  called the $k$-th
saddle value of        system (\ref{C-S});
the saddle  order of system
(\ref{C-S}) is defined to be the  minimal number $k$ such that
 $\pi_k \neq 0$.
\end{definition}

From \cite{Bibikov},  we know that
 $\pi_k$ is  a polynomial of the coefficients of
$P$ and $Q$.    Notice that
although the definition and form of $\pi_k$ are not entirely same as $L_k$ defined in Section 1,
the saddle orders defined in these two ways   are
the same. Moreover, the ideal generated by the first $k$ values of $\{\pi_k\}$ also coincides with $I_{k}$.
Therefore, in these senses, $\pi_k$ and $L_k$ are equivalent to each other.
From now on, we shall denote $\pi_k$ as  $L_k$.

We consider the case of homogeneous nonlinearity    first.
When the nonlinear part of system (\ref{C-S}) only consists of homogeneous polynomials of degree $n$,
then according to normal form theory, the saddle values of
system (\ref{C-S})  satisfy the following properties.

\begin{lemma}\label{Homo}
If $P$ and $Q$ in system (\ref{C-S}) are homogeneous polynomials of degree $n$, then the saddle values $L_k$ satisfy that
$$L_k=0, \qquad n_1 \nmid k, $$
where   $n_1=(n-1)/d$
and $d=(n-1, p+q)$.
\end{lemma}

\begin{proof} It is easy to check that  the normal form change of
coordinates has the form
$$ X=x+\sum_{i\geq1}\Phi_{1+i(n-1)}(x,y),\quad
Y=y+\sum_{i\geq1}\Psi_{1+i(n-1)}(x,y), $$
where each term in the summation is homogeneous
with the degree given by the subscript.
In each  step of normal form change of
coordinates, system (\ref{C-S}) takes the form
$$\dot X=pX+\sum_{i\geq1}\tilde\Phi_{1+i(n-1)}(X,Y),\quad  \dot Y=-qY+ \sum_{i\geq1}\tilde \Psi_{1+i(n-1)}(X,Y),   $$
where $\tilde\Phi_{1+i(n-1)}$ and $\tilde \Psi_{1+i(n-1)}$ represent the homogeneous polynomials of degree $1+i(n-1)$.
If $L_k\not=0$, then there must exist $i\geq 1$ so that
$$1+k(p+q)=1+i(n-1).$$
Denote by $q_1=(p+q)/d$. One must have that $q_1\mid i$ and $n_1 \mid k$.
\end{proof}

Now, we consider the following 1-parameter perturbed polynomial systems,
 \begin{eqnarray}\label{pertsys1}
\dot{x}=px+\varepsilon  pP(x,y), \quad
\dot{y}=-qy+\varepsilon  qQ(x,y),
\end{eqnarray}
where $P=\sum_{i+j\geq 2}p_{i,j}x^iy^j$, $Q=\sum_{i+j\geq 2}q_{i,j}x^iy^j$ and
$\varepsilon$ is a small parameter.

\begin{lemma}\label{lsaddlevalue}
The $k$-th saddle value $L_k$ of system (\ref{pertsys1}) is a polynomial of $\varepsilon$ and can be written as
$$L_k=(c_k+d_k)\varepsilon +o(\varepsilon),$$
where $c_k=p_{1+kq, kp}$ is the coefficient of term $x(x^qy^p)^k$ of $P(x,y)$ and $d_k=q_{kq,1+kp}$ is the coefficient of term $y(x^qy^p)^k$ of $Q(x,y)$.
\end{lemma}

\begin{proof} If $X=x+o(x,y), \ Y=y+o(x,y)$ is the change to normal form, then the normal form is
given by
$$\dot X=pX\biggl(1+\sum_{k=1}^{+\infty}A_{k}(\varepsilon)(X^qY^p)^k
\biggr) ,\quad \dot Y=-qY\biggl(1+ \sum_{k=1}^{+\infty}B_{k}(\varepsilon)  (X^qY^p)^k        \biggr),$$
where $A_{k}(\varepsilon)-B_{k}(\varepsilon)$ are the saddle values,
which, according to
Lemma 2.3 of \cite{IP},
are  quasi-homogeneous polynomials of
degree $k(p+q)$ of the coefficients $\varepsilon p_{i-1,j}$  and $\varepsilon q_{i,j-1}, i+j\leq k(p+q)+1$,  with weight $i+j-1$ of the original system.  Notice that here we adopt a conventional definition
of    quasi-homogeneous polynomials. Namely,
 a
polynomial $P(z_1,z_2,\cdots, z_N)$ is quasi-homogeneous of degree $m$ with weight $\alpha=(\alpha_1,\alpha_2,\cdots,\alpha_N)$
for the variables $z_1,\dots, z_N$ if
$$P(t^\alpha_1 z_1, t^\alpha_2z_2,\cdots, t^\alpha_Nz_N)=t^mP(z_1,z_2,\cdots, z_N).$$

In our case, $A_{k}(\varepsilon)$ ($B_{k}(\varepsilon)$, resp.)  is of degree $1$ of
the coefficient of $x(x^qy^p)^k$ of $P(x,y)$ ($y(x^qy^p)^k$ of $Q(x,y)$, resp.)  and of degree higher than $1$ in $\varepsilon p_{i,j}$
and $\varepsilon q_{i,j}$  with $i+j\leq k(p+q)$. We have
$$A_{k}(\varepsilon)=p_{1+kq, kp}\varepsilon+o(\varepsilon), \quad  B_{k}(\varepsilon)=-q_{kq,1+kp}\varepsilon+o(\varepsilon),$$
since they are coefficients of resonant items and the coefficients of lower degree do not effect their linear part.
\end{proof}

Consider the following system
\begin{equation}\label{pertsys}
\begin{array}{ll}
\dot{x}=px (1-U(x,y))+\varepsilon  pP(x,y), \\     \\
\dot{y}=-q y (1-U(x,y))+\varepsilon  qQ(x,y),
\end{array}
\end{equation}
where  $U$ is a polynomial without constant term, $P$ and $Q$
have nonlinear terms only.
Since  system (\ref{pertsys}) can be transformed to the system
\begin{equation}\label{pertsys2}
\begin{array}{ll}\dot{x}=px+\varepsilon  \frac{pP(x,y)}{1-U(x,y)}=px+\varepsilon  pP(x,y)\sum_{i=0}^{\infty}U^i, \\       \\
\dot{y}=-q y+\varepsilon  \frac{qQ(x,y)}{1-U(x,y)}=-qy+\varepsilon  qQ(x,y)\sum_{i=0}^{\infty}U^i, \end{array}
\end{equation}
 we can use the saddle values  of system (\ref{pertsys2}) instead of the saddle values  of system (\ref{pertsys}).

\begin{corollary}\label{lineorder}
The $k$-th saddle value $L_k$ of system (\ref{pertsys2}) can be written as
$$L_k=L_k(1)\varepsilon+o(\varepsilon),$$
where $L_k(1)$ is the coefficient of term $(x^qy^p)^{k+1}$ of  the power series$$\biggl(x^{q-1} y^p P+x^q y^{p-1} Q\biggr) \sum_{i=0}^{\infty}U^i.$$
\end{corollary}
\begin{proof}
By Lemma \ref{lsaddlevalue}, the
$k$-th saddle value $L_k$ of system (\ref{pertsys2}) can be written as
$$L_k=L_k(1)\varepsilon+o(\varepsilon),$$
where $L_k(1)$ is the sum of the coefficient of term $x(x^qy^p)^k$ of $P \sum_{i=0}^{\infty}U^i$ and
the coefficient of term $y(x^qy^p)^k$ of $Q \sum_{i=0}^{\infty}U^i$.
Clearly this sum is equal to the coefficient of term $(x^qy^p)^{k+1}$ of  $(x^{q-1} y^p P+x^q y^{p-1} Q) \sum_{i=0}^{\infty}U^i$.
\end{proof}

\section{Proof of Theorems}\label{proofth}
Below we consider systems (\ref{pertsys2})
with homogeneous $U$, $P$ and $Q$, where $P$ and $Q$ are chosen as
\begin{eqnarray}\label{choosePQ}
 P=\xi_{n+2} y^{n}, \quad Q=\sum_{j=1}^{n+1}\xi_{j} x^{n-j+1}y^{j-1}.
 \end{eqnarray}

Clearly the conclusion of Lemma \ref{Homo} is also valid for system (\ref{pertsys2}),
i.e., the saddle value $L_k=0$ when $n_1\nmid k$. Hence we only need to consider the saddle values
 $L_{mn_1}, m>0$.
 Notice that  all of them are polynomials of variables $\varepsilon$ and $\xi_{1},...,\xi_{n+2}$.
By Corollary \ref{lineorder}, $L_{mn_1}$  can be written as
$$L_{mn_1}=L_{mn_1}(1)\varepsilon+o(\varepsilon), $$
where $L_{mn_1}(1)$ is the coefficient of term $(x^{q}y^{p})^{mn_1+1}$ of the polynomial
\begin{eqnarray}\label{L1kPoly}
\biggl(\sum_{j=1}^{n+2}\xi_{j}x^{q+n-j+1}y^{p+j-2}\biggr) U^{i_{m}},
 \qquad i_{m}=\frac{m(p+q)}{d}-1.
\end{eqnarray}

It is easy to see that
 $L_{mn_1}(1)$ of system (\ref{pertsys2}) is a linear combination of  $\xi_1, \dots, \xi_{n+2}$,
 $$ L_{mn_1}(1)=\sum_{j=1}^{n+2}a_{mj}\xi_j,$$
 where $a_{mj}$ is the coefficient of term $x^{qmn_1-n-1+j}y^{pmn_1+2-j}$ of $U^{i_{m}}$.
If we let $\mathbf{L}$,
$\mathbf{L}_{1}$ and $\boldsymbol{\xi}$ be the column vectors consisting of $\{L_{mn_1}\}$, $\{L_{mn_1}(1)\}$ and $\{\xi_{m}\}, 1\leq m\leq n+1$, respectively,
 i.e.,
\begin{eqnarray*}
&&\mathbf{L}=(L_{n_1}, L_{2n_1}, \cdots, L_{(n+1)n_1})^{\top}, \\
&&\mathbf{L}_{1}=(L_{n_1}(1), L_{2n_1}(1), \cdots, L_{(n+1)n_1}(1))^{\top},\\
&&\boldsymbol{\xi}=(\xi_{1}, \xi_{2},  \cdots, \xi_{n+2})^{\top},
\end{eqnarray*}
then $\mathbf{L_{1}}=\mathbf{A} \boldsymbol{\xi},$ where $\mathbf{A}=(a_{mj})$ is an $(n+1)\times(n+2)$ matrix, and $$\mathbf{L}=\varepsilon \mathbf{L_1}+o(\varepsilon)=\varepsilon \mathbf{A} \boldsymbol{\xi}+o(\varepsilon).$$

For system (\ref{pertsys2}) with $P$ and $Q$ having the form   in (\ref{choosePQ}), we have the following  important    lemma.
The proof of this lemma is quite technical,
and we shall put it  at the end of the paper.

\begin{lemma}\label{TMn0matrix}
For sufficiently big $n$, there exists a homogeneous polynomial $U$ of degree $n-1$ such that    $Rank(\mathbf{A})$,
the rank of $\mathbf{A}$,  is $n+1$.
\end{lemma}


\subsection{Proof of Theorem \ref{TM1} and Theorem \ref{TM2}} 
\begin{proof}[Proof of Theorem \ref{TM1}]
To prove Theorem \ref{TM1},
it is sufficient to show that
there exists a system (\ref{pertsys2}) such that
its saddle order is $\frac{n^2-1}{d}$.

In  system (\ref{pertsys2}), we let $P$ and $Q$
take the forms as given in (\ref{choosePQ}). By Lemma \ref{TMn0matrix}, we can
 choose  a homogeneous $U$ of degree $n-1$
such that $Rank(\mathbf{A})=n+1.$ Thus for $\mathbf{L_1}=(0,0,\cdots, 0,1)^{\top}$, the equation $\mathbf{L_{1}}=\mathbf{A} \boldsymbol{\xi}$
has a nonzero solution
$$\boldsymbol{\xi}=(\xi_1^*, \xi_2^*, \cdots, \xi_{n+2}^*)^{\top}.$$

Denote by $\mathbf{L'}=\mathbf L/\varepsilon$, then $\mathbf{L'}=\mathbf{A} \boldsymbol{\xi}+o(1)$.
Clearly the Jacobian matrix $\frac{\partial \mathbf{L'}}{\partial \boldsymbol{\xi}}|_{\varepsilon=0}=\mathbf{A}$,
which is a row full rank matrix. Then by implicit function theorem, there exists functions $\xi_1(\varepsilon), \xi_2(\varepsilon), \cdots,
\xi_{n+2}(\varepsilon)$, analytic in $\varepsilon$,  so that
$$ \mathbf{L'}=(0,0,\cdots, 0,1)^{\top},$$
and  $(\xi_1(0), \xi_2(0), \cdots,
\xi_{n+2}(0))=(\xi_1^*, \xi_2^*, \cdots, \xi_{n+2}^*).$ Thus, for sufficiently small  $\varepsilon >0$, we have
$$\mathbf L=\varepsilon \mathbf{L'}=(0,0,\cdots, 0,\varepsilon)^{\top}.$$
In other words,
\begin{eqnarray*}\label{fchz}
\begin{cases}
L_{mn_1}=0,   \quad  m=1,2,...,n,\\     \\
L_{(n+1)n_1}=\varepsilon \neq 0.
\end{cases}
\end{eqnarray*}
This implies that there at least exists a saddle system with  homogeneous nonlinearities of degree $n$
such that its saddle order is $\frac{(n+1)(n-1)}{d}=\frac{n^2-1}{d}$.
Then Theorem \ref{TM1} is proved.
\end{proof}

\begin{proof}[Proof of Theorem \ref{TM2}]
For any sufficiently big $n$, suppose that $r\equiv n\,\, \mbox{mod}\  p+q$, where $0\leq r\leq p+q-1$, and denote by $n'=n-r$,
 then $(n'-1, p+q)=1$. By Theorem 1, there exists a system
 $$\dot x=px+P(x,y), \qquad  \dot y=-qy+Q(x,y),$$
 where $P$ and $Q$ are homogeneous polynomials of degree $n'-1$, so that the order of the origin of this system  is
 $n'^2-1=n^2-2rn+r^2-1$.

 Obviously, the following system
 $$\dot x=(px+P(x,y))(1+x^r), \qquad  \dot y=(-qy+Q(x,y))(1+x^r),$$
is a system of degree $n$ and the origin is a $p:-q$ saddle of order $n^2-2rn+r^2-1$.
So $M(p,q,n)\geq n^2-2rn+r^2-1$, the proof is finished.
\end{proof}

\subsection{Proof of Theorem \ref{TMideal} and Theorem \ref{TMideal2}}

Before giving a detailed proof of these two theorems, we need the following lemma.
\begin{lemma}\label{p'}
For any given positive integers $n$ and $p$, there exists a  positive integer $p'$, which is a factor of $p$,  such that $(n-1-p', pp')=1$.
\end{lemma}
\begin{proof}
If $p=1$, obviously we take $p'=1$.
For $p>1$, suppose that
$$p=\prod_{i=1}^{k} p_i^{r_i},$$
where $p_1, p_2, \cdots, p_k$ are different prime numbers.
If all of these $p_{i}$ are also the factors of $n-1$, then we take $p'=1$.
For other cases,  we assume that, without loss of generality,
$p_1, p_2, \cdots, p_m$ are the factors of $n-1$ and $p_{m+1}, p_{m+2}, \cdots, p_k$ are not. 
Let $p'=\prod_{i=m+1}^{k} p_i$.
It is straightforward to check that $(n-1-p', pp')=1$.
\end{proof}

Now  choose  $p'$  such that it is a factor of $p$ and $(n-1-p',pp')=1$, and denote by  $N_{1}=n-1-p'$.
Note that $(p,q)=1$ and $(p+q,n-1)=d$, we have the following corollary immediately.

\begin{corollary}\label{cor2}
$(N_{1},d)=1$, and $(pp',d)=1$.
\end{corollary}

Since $(N_1, pp')=1$, for any $j\in \mathbb N$, there exists an integer $s_j$, $0<s_j\leq N_1$, so that
$s_j pp' \equiv d(j-2)$ mod $ N_{1}$. Obviously, if $1\leq j\leq 3+p', j\not=2$, then $0<s_j<N_1$.
Denote by $N_{2}=\max\{s_j|1\leq j\leq 3+p', j\not=2\}$, then $0<N_2<N_1$.

\begin{lemma}\label{N1N2}
$N_{2}\geq N_{1}/2$.
In particular, if $p=1$, then $N_{1}=n-2$ and $N_{2}=n-2-d$.
\end{lemma}

\begin{proof}

Notice that $ s_j pp' \equiv d(j-2)$ mod $N_{1}$, so
 $s_{1}+s_{3} \equiv 0$ mod $N_{1}$. On the other hand, $0<s_{1}+s_{3}< 2N_{1}$,
we have  $s_{1}+s_{3}= N_{1}$, which implies one of $s_{1}$ and $s_{3}$ is bigger than $N_{1}/2$.
By the definition of $N_{2}$, we have that $N_{2}\geq N_{1}/2$.

If $p=1$, then $p'=1$ and $N_{1}=n-2$.
Note that $s_{2}=N_{1}$ and $s_{j}-s_{2} \equiv d(j-2)$  mod $N_{1}$ for all $0< j \leq 3+p',\ j\neq 2$,
thus $s_j\leq N_{1}-d$, i.e., $N_{2}\leq N_{1}-d$.
Besides, it is easy to see that $s_{1}=N_{1}-d$, so $N_{2}= N_{1}-d$.
\end{proof}

\begin{proof}[Proof of Theorem \ref{TMideal}]
We shall show that, there exists a system of form (\ref{pertsys2}) with homogenous polynomials $P, Q, U$ of degree $n, n, n-1$ respectively,
 such that its saddle value
 $L_{(N_{1}+N_{2})n_1}$ does not belong to the
ideal $I_{(N_{1}+N_{2})n_1 -1}=I_{(N_{1}+N_{2}-1)n_1}$.

We take  $U=x^{p'}(x^{N_{1}}+y^{N_{1}})$ and choose
 $P$ and $Q$ having the form as in (\ref{choosePQ}).
Then we claim that
$$L_{(N_{1}+N_{2})n_1}\not\in I_{(N_{1}+N_{2}-1)n_1}.$$

To prove the claim, we suppose otherwise.
That is, there exist $N_{1}+N_{2}-1$ polynomials $F_{m}$, $m=1,...,N_{1}+N_{2}-1$ in $\varepsilon,\xi_{1},...,\xi_{n+2}$
such that
\begin{eqnarray*}\label{polyrep}
L_{(N_{1}+N_{2})n_1}=\sum_{m=1}^{N_{1}+N_{2}-1}F_{m} L_{mn_1}.
\end{eqnarray*}
By Corollary \ref{lineorder}, $L_{mn_1}$  can be written as
$L_{mn_1}=L_{mn_1}(1)\varepsilon+o(\varepsilon),$  where $L_{mn_1}(1)$  linearly depends on $\xi_{1},...,\xi_{n+2}$.
Comparing the coefficient of  $\varepsilon$, we have
\begin{eqnarray}\label{polyrep2}
L_{(N_{1}+N_{2})n_1}(1)=\sum_{m=1}^{N_{1}+N_{2}-1}F_{m,0} L_{mn_1}(1),
\end{eqnarray}
where $F_{m,0}$, which is from  $F_{m}$'s terms and is independent of $\varepsilon$, is also a polynomial in $\xi_{1},...,\xi_{n+2}$.
By (\ref{L1kPoly}) and $(N_{1},pp')=1$, after the straightforward calculation,
we obtain
\begin{eqnarray}\label{A0}
L_{mn_1}(1)=
\begin{cases}
C_{i_{m}}^{l_{m}} \xi_{j_{m}}+
C_{i_{m}}^{l_{m}-1} \xi_{N_{1}+j_{m}}, \quad \ \  1\leq j_{m}\leq 3+p'\\
\\
C_{i_{m}}^{l_{m}} \xi_{j_{m}}, \qquad \qquad\ \ \ \ \qquad \quad 3+p'< j_{m}\leq N_{1}
\end{cases},
\end{eqnarray}
where $j_{m}$ and $l_{m}$ satisfy
\begin{eqnarray}\label{jmlm}
d(j_{m}-2)&\equiv & mpp' \ \ \mbox{mod}\  N_{1},\ \ 1\leq j_{m}\leq N_{1}, \nonumber\\
\\      \nonumber
dl_{m}&=&mp+\frac{mpp'-d(j_{m}-2)}{N_{1}}.
\end{eqnarray}
According to the definition of $N_{2}$, we have $0<j_{N_{2}}\leq 3+p'$.

Since the equation (\ref{polyrep2}) holds for any values of $\xi_{1},...,\xi_{n+2}$,
it certainly holds for $\xi_{m}=0$ except $m=j_{N_{2}}$ and $m=N_{1}+j_{N_{2}}$, i.e.,
the following equation holds for any values of $\xi_{j_{N_{2}}}$ and $\xi_{N_{1}+j_{N_{2}}}$,
\begin{eqnarray*}\label{polyrep3}
C_{i_{N_{1}+N_{2}}}^{l_{N_{1}+N_{2}}} \xi_{j_{N_{2}}}+
C_{i_{N_{1}+N_{2}}}^{l_{N_{1}+N_{2}}-1} \xi_{N_{1}+j_{N_{2}}}=\widetilde{F}_{N_{2},0} (C_{i_{N_{2}}}^{l_{N_{2}}} \xi_{j_{N_{2}}}+
C_{i_{N_{2}}}^{l_{N_{2}}-1} \xi_{N_{1}+j_{N_{2}}}),
\end{eqnarray*}
where $\widetilde{F}_{N_{2},0}(\xi_{j_{N_{2}}},\xi_{N_{1}+j_{N_{2}}})=F_{N_{2},0}|_{\xi_{m}=0,m\neq j_{N_{2}},N_{1}+j_{N_{2}}}$.
The above equation implies that there exists a nonzero constant $\beta$ such that
\begin{eqnarray*}
C_{i_{N_{1}+N_{2}}}^{l_{N_{1}+N_{2}}} \xi_{j_{N_{2}}}+
C_{i_{N_{1}+N_{2}}}^{l_{N_{1}+N_{2}}-1} \xi_{N_{1}+j_{N_{2}}}=\beta (C_{i_{N_{2}}}^{l_{N_{2}}} \xi_{j_{N_{2}}}+
C_{i_{N_{2}}}^{l_{N_{2}}-1} \xi_{N_{1}+j_{N_{2}}}).
\end{eqnarray*}
Thus
\begin{eqnarray*}
0=\frac{C_{i_{N_{1}+N_{2}}}^{l_{N_{1}+N_{2}}}}{C_{i_{N_{1}+N_{2}}}^{l_{N_{1}+N_{2}}-1}}- \frac{C_{i_{N_{2}}}^{l_{N_{2}}}}{C_{i_{N_{2}}}^{l_{N_{2}}-1}}
&=&\frac{i_{N_{1}+N_{2}}+1}{l_{N_{1}+N_{2}}}-\frac{i_{N_{2}}+1}{l_{N_{2}}}\\
&=&\frac{(p+q)((dl_{N_{2}}-pN_{2})N_{1}-N_{2}pp_{0})}{d^{2}l_{N_{2}}l_{N_{1}+N_{2}}}.
\end{eqnarray*}
On the other hand, since $0<N_{2}<N_{1}$,     we have
$$(dl_{N_{2}}-pN_{2})N_{1}-N_{2}pp'\equiv -N_{2}pp' \not \equiv 0 \  \ \mbox{mod}\  N_{1},$$
which leads to a contradiction.

Finally, by Lemma \ref{N1N2}, we have
\begin{eqnarray*}
&&M^{I}_{h}(p,q,n)\geq (N_{1}+N_{2})n_{1}\geq \frac{3}{2d} n^{2}-\frac{3}{2d}((2+p')n-1-p'),\\
&&M^{I}_{h}(1,q,n)\geq (N_{1}+N_{2})n_{1}=  \frac{2}{d} n^{2}-\frac{1}{d}((6+d)n-4-d).
\end{eqnarray*}
Thus the proof is done.
\end{proof}

\begin{proof}[Proof of Theorem \ref{TMideal2}]
The proof of this theorem essentially  takes
the same pattern as
the proof of Theorem \ref{TM2}. Therefore we omit the details.
\end{proof}

\subsection{ Proof of Lemma \ref{TMn0matrix}}

In this subsection, we
give a proof of  Lemma \ref{TMn0matrix}.

Let
$$U=f(x,y)+\mu  g(x,y),\qquad  0<\mu\ll 1,$$ where
 $f=x^{p'} (x^{N_{1}}+y^{N_{1}})$,
and $g$ is a homogeneous polynomial of degree $n-1$
 to be determined later. Here $p'$ and $N_1$ are the same as $p'$ and $N_1$ in Subsection 3.2.

\begin{proof}[Proof of Lemma \ref{TMn0matrix}]
Substituting $U=f(x,y)+\mu  g(x,y)$ into (\ref{L1kPoly}), we obtain
\begin{eqnarray}\label{L1kPoly2}
\biggl(\sum_{j=1}^{n+2}\xi_{j}x^{q+n-j+1}y^{p+j-2}\biggr)\sum_{k=0}^{i_{m}}\mu^{k} C^{k}_{i_{m}} f^{i_{m}-k} g^{k}.
\end{eqnarray}
Then the matrix $\mathbf{A}$ can be written as
$$\mathbf{A}=\sum_{k=0}^{i_{m}}\mu^{k}\mathbf{A}_{k},$$
where $\mathbf{A}_{k}=(a_{k,mj})$ is an $(n+1)\times(n+2)$ matrix and
$a_{k,mj}$ is the coefficient of term $x^{qmn_1-n-1+j}y^{pmn_1+2-j}$ of $C^{k}_{i_{m}} f^{i_{m}-k} g^{k}$.

For $\mathbf{A}_{0}$, from (\ref{A0}) we have
\begin{eqnarray*}
a_{0,mj}=
\begin{cases}
0, \qquad \quad \, j\neq j_{m}, j_{m}+N_{1} \\     \\
C_{i_{m}}^{l_{m}}, \qquad j=j_{m}\\               \\
C_{i_{m}}^{l_{m}-1}, \quad j=j_{m}+N_{1} \leq n+2
\end{cases}.
\end{eqnarray*}
For convenience, we divide $\mathbf{A}_{k}$ into four blocks, i.e.,
\begin{eqnarray*}
\mathbf{A}_{k}=
\left(
  \begin{array}{cc}
    \mathbf{A}_{k,11} & \mathbf{A}_{k,12} \\      \\
    \mathbf{A}_{k,21} & \mathbf{A}_{k,22} \\
  \end{array}
\right),
\end{eqnarray*}
where $\mathbf{A}_{k,11}$ is comprised of the first $N_{1}$ rows and the first $N_{1}$ columns of $\mathbf{A}_{k}$.
Clearly $\mathbf{A}_{k,11}$ is a $N_1\times N_1$ square matrix,
and there is exactly one nonzero element $a_{0,mj_{m}}$ in each row and each column of $\mathbf{A}_{0,11}$,
therefore  the determinant of  $\mathbf{A}_{0,11}$
$$\det{(\mathbf{A}_{0,11})}=\pm \prod_{m=1}^{N_{1}}a_{0,mj_{m}} \neq 0.$$

We first consider the case $(p-1)p'>1$ or $d> 2+p'$.
In this case, by Corollary \ref{cor2} and (\ref{jmlm}),
 for $1\leq m \leq 2+p'$ we have
$3+p'<j_{m}\leq N_{1}.$
This implies $\mathbf{A}_{0,22}=0$. Now we need to consider $\mathbf{A}_{1,22}$.

In this case, we let
\begin{eqnarray*}
g=\sum_{m=1}^{1+p'}x^{n-1-\delta_{m}}y^{\delta_{m}}+x^{n-1-\delta_{2+p'}}y^{\delta_{2+p'}},
\end{eqnarray*}
where $\delta_{m}=j_{m}-(p'+4-m)$ for $1\leq m \leq 1+p'$ and $\delta_{2+p'}=j_{2+p'}-1$.
It is easy to see that  $0<\delta_{m}< N_{1}$.


By (\ref{L1kPoly2}), it is not hard to check that,
for $1\leq m \leq 2+p'$ and $1\leq j \leq 3+p'$,
the element $a_{1, (N_{1}+m)(N_{1}+j)}\neq 0$ only when
there exists a number $1\leq m' \leq 2+p'$ such that
$d(j+\delta_{m'}-2)\equiv mpp'$ mod $N_{1}$,  
which requires that $m'$ must be equal to $m$,
and $j=p'+4-m$ for $1\leq m \leq 1+p'$ and $j=1$ for $m = 2+p'$.
Therefore, for $m\in \{1,...,2+p'\} \cup \{N_{1}+1,...,N_{1}+2+p'\}$, and
$j\in \{1,..., 3+p'\} \cup \{N_{1}+1,...,N_{1}+3+p'\}$,
we have
\begin{eqnarray*}
a_{1,mj}=
\begin{cases}
0, \ \ \ \ \ \quad \ \ \ j\neq j'_{m}, \ N_{1}+j'_{m}\\       \\
i_{m}C_{i_{m}-1}^{l_{m}}, \ \  j=j'_{m} \\                  \\
i_{m}C_{i_{m}-1}^{l_{m}-1}, \ \ j=N_{1}+j'_{m}
\end{cases},
\end{eqnarray*}
where $j'_{m}=j'_{N_{1}+m}=p'+4-m$ for $1\leq m \leq 1+p'$ and $j'_{2+p'}=j'_{N_{1}+2+p'}=1$.
Obviously $j'_{m}\neq 2$.

Denote by $\mathbf{A}'=\mathbf{J}_{1}\mathbf{A}\mathbf{J}_{2}$, where
\begin{eqnarray*}
\mathbf{J}_{1}=
\left(
  \begin{array}{cc}
    \mathbf{I}_{N_{1}} & 0 \\     \\
    -\mathbf{A}_{0,21}\mathbf{A}_{0,11}^{-1} & \mathbf{I}_{2+p'} \\
  \end{array}
\right), \                          \quad
\mathbf{J}_{2}=
\left(
  \begin{array}{cc}
    \mathbf{I}_{N_{1}} & -\mathbf{A}_{0,11}^{-1}\mathbf{A}_{0,12} \\  \\
    0 & \mathbf{I}_{3+p'} \\
  \end{array}
\right),
\end{eqnarray*}
and $\mathbf{I}_{k}$ means the  identity matrix of order $k$.
Then we have
\begin{eqnarray*}
&&\mathbf{A}'=
\left(
  \begin{array}{cc}
    \mathbf{A}_{0,11} & 0 \\    \\
    0 & 0 \\
  \end{array}
\right)
+\mu
\left(
  \begin{array}{cc}
    \mathbf{A}_{1,11} & \mathbf{A}_{1,12}' \\
    &\\
    \mathbf{A}_{1,21}' &
    \mathbf{A}_{1,22}'
  \end{array}
\right)
+\sum_{k=2}^{i_{m}}\mu^{k}\mathbf{A}'_{k},
\end{eqnarray*}
where
\begin{eqnarray*}
\mathbf{A}_{1,12}'&=& \mathbf{A}_{1,12}-\mathbf{A}_{1,11}\mathbf{A}_{0,11}^{-1}\mathbf{A}_{0,12}, \\
 \mathbf{A}_{1,21}'&=&   \mathbf{A}_{1,21}-\mathbf{A}_{0,21}\mathbf{A}_{0,11}^{-1}\mathbf{A}_{1,11}, \\
 \mathbf{A}_{1,22}'&= &  \mathbf{A}_{1,22}-\mathbf{A}_{0,21}\mathbf{A}_{0,11}^{-1}\mathbf{A}_{1,12}
    -(\mathbf{A}_{1,21}-\mathbf{A}_{0,21}\mathbf{A}_{0,11}^{-1}\mathbf{A}_{1,11})\mathbf{A}_{0,11}^{-1}\mathbf{A}_{0,12},\\
\mathbf{A}_{k}'&=&\mathbf{J}_{1}\mathbf{A}_{k}\mathbf{J}_{2},\ \ k \geq 2.
\end{eqnarray*}
This process can be described as a series of the following explicit elementary transformations on $\mathbf{A}$:

(i). add row $m$ multiplied by a scalar $-\frac{a_{0,(N_{1}+m)j_{m}}}{a_{0,mj_{m}}}$ to row $N_{1}+m$, $m=1,2,...,2+p'$;

(ii). add column $k$ multiplied by a scalar $-\frac{a_{0,(m^{*}_{k})(N_{1}+k)}}{a_{0,(m^{*}_{k})k}}$ to column $N_{1}+k$, $1\leq k \leq 3+p'$,
where $m^{*}_{k}\in [1, N_{1}]$ is  an integer satisfying $j_{m^{*}_{k}}=k$.

Then we obtain the elements $a'_{1,(N_{1}+m)(N_{1}+j)}$ of $\mathbf{A}_{1,22}'$ as follows,
for $1\leq m\leq 2+p'$ and $1\leq j \leq 3+p'$, noticing that $j'_{m}\neq 2$,
\begin{eqnarray*}
a'_{1,(N_{1}+m)(N_{1}+j)}=
\begin{cases}
0, \ \ \ \ \ \qquad \ \qquad \qquad \qquad \qquad  \ \ \   \ \ j\neq j'_{m} \\                \\
\frac{N_{1}(p+q)(j'_{m}-2)}{(qN_{1}-pp')m^{*}_{j'_{m}}+d(j'_{m}-2)}C^{l_{N_{1}+m}}_{i_{N_{1}+m}}\neq 0,\ \ \  j=j'_{m}
\end{cases}.
\end{eqnarray*}

Denote by $\widetilde{\mathbf{A}}'$(resp. $\widetilde{\mathbf{A}}'_{k,22}$, $\widetilde{\mathbf{A}}'_{k,12}$)
the submatrix of $\mathbf{A}'$(resp. $\mathbf{A}'_{k,22}$, $\mathbf{A}'_{k,12}$)
by removing column $N_{1}+2$(resp. the second column).
Then $\widetilde{\mathbf{A}}'$ and $\widetilde{\mathbf{A}}'_{1,22}$ both are square matrices, and each row and each column of $\widetilde{\mathbf{A}}'_{1,22}$  both have exactly one nonzero element. Thus
\begin{eqnarray*}
\det(\widetilde{\mathbf{A}}'_{1,22})&=&
\pm \prod_{m=1}^{2+p'}a'_{1,(N_{1}+m)(N_{1}+j'_{m})} \neq 0,\\
\det(\widetilde{\mathbf{A}}')&=&
\mu^{2+p'} \det
\left(
  \begin{array}{ll}
    \mathbf{A}_{0,11}+\mu (\mathbf{A}_{1,11}+\cdots) &  0+\mu (\widetilde{\mathbf{A}}'_{1,12}+\cdots)\\
    \mathbf{A}_{1,21}'+\mu(\mathbf{A}_{2,21}'+\cdots) &  \widetilde{\mathbf{A}}'_{1,22}+\mu(\widetilde{\mathbf{A}}'_{2,22}+\cdots)
  \end{array}
\right)\\
&=&\mu^{2+p'} (\det(\mathbf{A}_{0,11})\det(\widetilde{\mathbf{A}}'_{1,22})+\mu h(\mu)),
\end{eqnarray*}
where $h(\mu)$ is a polynomial in $\mu$.
Noticing that $\det(\mathbf{A}_{0,11})\det(\widetilde{\mathbf{A}}'_{1,22})\neq 0$
and $\mu$ is a sufficiently small positive number,  we have $\det (\widetilde{\mathbf{A}}')\neq 0$,
which implies that $Rank (\mathbf{A})= Rank (\mathbf{A}')=Rank (\widetilde{\mathbf{A}}')=n+1$.

Next we consider the case that $pp'+2\leq 3+p'$ and $d\leq 2+p'$ hold simultaneously,
which contain only finite possibilities. In this case, the above process is still valid,
as long as we can find a suitable $g$.
The corresponding $g$ of each case
is listed in the following table(note that $(pp',d)=1$):
\\
\\
\begin{tabular}{|l|l|}
\hline
$p=p'=1,\  d=1$    &    $g=x^{n-5}y^{4}$    \\
\hline
$p=p'=1,\ d=2$    &    $g=x^{n+3-j_{1}}y^{j_{1}-4}+x^{n-j_{3}}y^{j_{3}-1}$   \\
\hline
$p=p'=1, \ d=3$    &    $g=x^{n+3-j_{1}}y^{j_{1}-4}+x^{n-j_{2}}y^{j_{2}-1}$  \\
\hline
$p=2,\ p'=1,\  d=1$    &    $g=x^{n-4}y^{3}+x^{n-8}y^{7}$     \\
\hline
$p=2,\ p'=1,\  d=3$    &    $g=x^{n-j_{1}-4}y^{j_{1}-3}+x^{n-j_{2}}y^{j_{2}-1}$     \\
\hline
else: $\mathbf{A}_{0,22}= 0$   &    $g=\sum_{m=1}^{1+p'}x^{n-1-\delta_{m}}y^{\delta_{m}}+x^{n-1-\delta_{2+p'}}y^{\delta_{2+p'}}$  \\
\hline
\end{tabular}
\\
\\
It is easy to check, no matter what kind of the cases in the table,
$\det (\widetilde{\mathbf{A}}')\neq 0$ always holds. The proof is finished.
\end{proof}

\subsection*{Acknowledgements}
G. Dong, C. Liu  and J. Yang are supported, resp., in part by the NSFC of China Grant 11701217, 11371269 and 11271026.

\end{document}